\documentclass[12pt]{amsart}

\usepackage[all]{xy} 
\usepackage{rotating}
\usepackage{amsmath}
\usepackage{units}
\usepackage{hyperref}
\usepackage{url}
\usepackage{framed}
\usepackage{framed}

\usepackage{latexsym}
\usepackage{amssymb}
\usepackage{color}
\usepackage{url}
\usepackage{fullpage}
\usepackage{comment}

\newcommand{\refeq}[1]{(\ref{#1})}  
\newcommand{\beq}[1]{\begin{equation}\label{#1}}
\newcommand{\eeq}{\end{equation}}
\newcommand{\diamms}[2]{{\mathbf{diam}}_{#2}\!\left({#1}\right)}  
\newcommand{\bfX}{\mathbf{X}}
\newcommand{\bfY}{\mathbf{Y}}
\newcommand{\bfW}{\mathbf{W}}
\newcommand{\dgro}[2]{d_{\mathcal G\mathcal H}({#1},{#2})}
\newtheorem{corollary}{Corollary}[section]
\newtheorem{theorem}{Theorem}[section]

\newtheorem{proposition}{Proposition}[section]

\theoremstyle{definition}

\newtheorem{remark}{Remark}[section]

\def\R{\mathbb{R}}
\def\N{\mathbb{N}}

 \newcommand{\pow}[1]{\mathrm{pow}({#1})}
\newcommand\dk[1]{\mathrm{D}_{#1}}
\newcommand\pers[1]{\mathrm{pers}(#1)}

\title{A distance between filtered spaces via tripods}

\author{Facundo M\'emoli}
\date{\today}
\email{ memoli@math.osu.edu}

\begin{document}

\address[F. M\'emoli]{Department of Mathematics and Department of
  Computer Science and Engineering, The Ohio State University. Phone: (614) 292-4975,
Fax: (614) 292-1479.}

\begin{abstract}
We present a simplified treatment of stability of filtrations on finite spaces. Interestingly, we can lift the stability result for combinatorial filtrations from \cite{vineyards} to the case when two filtrations live on different spaces without directly invoking the concept of interleaving. We then prove that this distance is intrinsic by constructing explicit geodesics between any pair of filtered spaces. Finally we use this construction to obtain a strengthening of the stability result. 
\end{abstract}

\thanks{This work was supported by NSF grants IIS-1422400 and CCF-1526513.} 

\maketitle

\setcounter{tocdepth}{3}
\tableofcontents

\section{Introduction}
The goal for the construction that I describe in this note was to lift the stability result of \cite{vineyards} to the setting when the simplicial filtrations are not necessarily defined on the same set. The ideas in this note were first presented at ATCMS in July 2012. A partial discussion appears in \cite{facundo2012banff}. 

Section \ref{sec:geodesic} proves that that this construction defines a geodesic metric on the collection of finite filtered spaces. There I give an improvement of the stability of persistence which uses these geodesics.

\section{Simplicial Homology}
Given a simplicial complex $L$ and simplices $\sigma,\tau\in L$, we write $\sigma\subseteq \tau$ whenever $\sigma$ is a face of $\tau$. For each integer $\ell\geq 0$ we denote by $L^{(\ell)}$ the $\ell$-skeleton of $L$.

Recall that given two finite simplicial complexes $L$ and $S$, a \emph{simplicial map} between them arises from any map $f:L^{(0)}\rightarrow S^{(0)}$ with the property that whenever $p_0,p_1,\ldots,p_k$ span a simplex in $L$, then $f(p_0),f(p_1),\ldots,f(p_k)$ span a simplex of $S$. One does not require that the vertices $f(p_0),f(p_1),\ldots,f(p_k)$ be all distinct.  Given a map $f:L^{(0)}\rightarrow S^{(0)}$ between the vertex sets of the finite simplicial complexes $L$ and $S$, we let $\overline{f}:L\rightarrow S$ denote the induced \emph{simplicial map}.

We will make use of the following theorem in the sequel.
\begin{theorem}[Quillen's Theorem A in the simplicial category, \cite{quillen}] Let $\zeta:S\rightarrow L$ be a simplicial map between two finite complexes. Suppose that the preimage of each closed simplex of $L$ is contractible. Then $\zeta$ is a  homotopy equivalence.
\end{theorem}

\begin{corollary}\label{coro:eq-pd}
Let $L$ be a finite simplicial complex and $\varphi:Z\rightarrow L^{(0)}$ be any  surjective map with finite domain $Z$. Let $S:=\{\tau\subseteq Z|\,\varphi(\tau)\in L\}$. Then $S$ is a simplicial complex and the induced simplicial map
$\overline{\varphi}:S\rightarrow L$ is an homotopy equivalence.
\end{corollary}

\begin{proof}
Note that $S = \bigcup_{\sigma\in L}\{\tau\subseteq Z|\,\varphi(\tau)=\sigma\}$ so it is clear that $S$ is a simplicial complex with vertex set $Z$. That the preimage of each $\sigma\in L$ is contractible is trivially true since those preimages are exactly the simplices in $S$. The conclusion follows directly from Quillen's Theorem A.
\end{proof}

In this paper we consider homology with coefficients in a field $\mathbb{F}$ so that  given a simplicial complex $L$, then for each $k\in\N$, $H_k(L,\mathbb{F})$ is a vector space. To simplify notation, we drop the argument $\mathbb{F}$ from the list and only write $H_k(L)$ for the homology of $L$ with coefficients in $\mathbb{F}$.

\section{Filtrations and Persistent Homology}

Let $\mathcal{F}$ denote the set of all finite \emph{filtered spaces}: that is pairs $\bfX=(X,F_X)$ where $X$ is a finite set and $F_X:\mathrm{pow}(X)\rightarrow \R$ is a monotone function. Any such function is called a \emph{filtration} over $X$. Monotonicity in this context refers to the condition that $F_X(\sigma)\geq F_X(\tau)$ whenever $\sigma \supseteq \tau.$ Given a finite set $X$, by $\mathcal{F}(X)$ we denote the set of all possible filtrations $F_X:\pow{X}\rightarrow \R$ on $X$. Given a filtered space $\mathbf{X}=(X,F_X)\in\mathcal{F}$, for each $\varepsilon\in \R$ define the simplicial complex 
$$L_\varepsilon(\mathbf{X}):=\big\{\sigma\subseteq X|\,F_X(\sigma)\leq \varepsilon\big\}.$$
One then considers the nested family of simplicial complexes

$$L(\mathbf{X}):=\big\{L_{\varepsilon}(\bfX)\subset L_{\varepsilon'}(\bfX)\}_{\varepsilon\leq \varepsilon'}$$

where each $L_{\varepsilon}(\bfX)$ is, by construction, finite. At the level of homology, for each $k\in \N$ the above inclusions give rise to a system of vector spaces and linear maps

$$\mathbb{V}_k(\bfX):=\big\{V_{\varepsilon}(\bfX)\stackrel{v_{\varepsilon,\varepsilon'}}{\longrightarrow} V_{{\varepsilon'}}(\bfX)\big\}_{\varepsilon\leq \varepsilon'},$$

which is called a \emph{persistence vector space.} Note that each $V_{\varepsilon}(\bfX)$ is finite dimensional.

Persistence vector spaces admit a \emph{classification up to isomorphism} in terms of collections of intervals so that to the persistence vector space  $\mathbb{V}$ one assigns a multiset of intervals $I(\mathbb{V})$ \cite{zz}. These collections of intervals are sometimes referred to as \emph{barcodes} or also \emph{persistence diagrams}, depending on the graphical representation that is adopted \cite{comptopo-herbert}. We denote by $\mathcal{D}$ the collection of all finite persistence diagrams. An element $D\in\mathcal{D}$ is a \emph{finite} multiset of points $$D= \{(b_\alpha,d_\alpha),\,0\leq b_\alpha\leq d_\alpha,\,\alpha\in A\}$$ for some (finite) index set $A$. Given $k\in\N$, to any filtered set $\bfX\in\mathcal{F}$ one can attach a persistence diagram via 
$$\bfX\longmapsto L(\bfX) \longmapsto \mathbb{V}_k(\bfX) \longmapsto I\big(\mathbb{V}_k(\bfX)\big).$$
We denote by $\dk{k}:\mathcal{F}\rightarrow \mathcal{D}$ the resulting composite map. Given $\bfX=(X,F_X)$,  we will sometimes write $\dk{k}{(F_X)}$ to denote $\dk{k}{(\bfX)}$.

\section{Stability of filtrations}
The \emph{bottleneck distance} is a useful notion of distance between persistence diagrams and we recall it's definition next. We will follow the presentation on \cite{carlsson_2014}. Let $\Delta\subset \R^2_+$ be comprised of those points which sit above the diagonal: $\Delta:=\{(x,y)|\,x\leq y\}.$

Define the \emph{persistence} of a point $P=(x_P,y_P)\in\Delta$ by  $\pers{P}:=y_P-x_P$. 

 Let $D_1=\{P_\alpha\}_{\alpha\in A_1}$ and $D_2=\{Q_\alpha\}_{\alpha\in A_2}$ be two persistence diagrams indexed over the finite index sets $A_1$ and $A_2$, respectively. Consider subsets $B_i\subseteq A_i$ with $|B_1|=|B_2|$ and any bijection $\varphi:B_1\rightarrow B_2$, then define
$$J(\varphi):=\max\left(\max_{\beta\in B_1}\|P_\beta-Q_{\varphi(\beta)}\|_\infty,\max_{\alpha\in A_1\backslash B_1}\frac{1}{2}\pers{P_\alpha},\max_{\alpha\in A_2\backslash B_2}\frac{1}{2}\pers{P_\alpha}\right).$$
Finally, one defines the bottleneck distance between $D_1$ and $D_2$ by $$d_{\mathcal{D}}(D_1,D_2):=\min_{(B_1,B_2,\varphi)} J(\varphi),$$
where $(B_1,B_2,\varphi)$ ranges over all $B_1\subset A_1$, $B_2\subset A_2$, and bijections $\varphi:B_1\rightarrow B_2$.

One of the standard results about the stability of persistent homology invariants, which is formulated in terms of the Bottleneck distance, is the proposition below which we state in  a weaker form that will suffice for our presentation:\footnote{In \cite{vineyards} the authors do not assume that the underlying simplicial complex is the full powerset.}
 \begin{theorem}[\cite{vineyards}]\label{theo:stab-vineyards}
 For all finite sets $X$ and filtrations $F,G:\mathrm{pow}(X)\rightarrow \R$, 
 $$d_{\mathcal{D}}(\dk{k}(F),\dk{k}(G))\leq \max_{\sigma\in\mathrm{pow}(X)}|F(\sigma)-G(\sigma)|,$$
 for all $k\in\N.$
 \end{theorem}

The proof of this theorem offered in \cite{vineyards} is purely combinatorial and elementary. This result requires that the two filtrations be given on the \emph{same} set. This restriction will be lifted using the ideas that follow.

\subsection{Filtrations defined over different sets}
 A \emph{parametrization} of a finite set $X$ is any finite set $Z$ and a surjective map $\varphi_X:Z\rightarrow X$.
Consider a filtered space $\bfX = (X,F_X)\in\mathcal{F}$ and a parametrization $\varphi_X:Z\rightarrow X$ of $X$. By $\varphi_X^\ast F_X$ we denote the  \emph{pullback filtration} induced by $F_X$ and the map $\varphi_X$ on $Z$. This filtration is given by $\tau\mapsto F_X(\varphi_X(\tau))$ for all $\tau\in\mathrm{pow}(Z).$

A useful corollary of the persistence homology isomorphism theorem \cite[pp. 139]{comptopo-herbert} and Corollary \ref{coro:eq-pd} is that the persistence diagrams of the original filtration and the pullback filtration are identical.
\begin{corollary}\label{coro:same}
Let $\bfX=(X,F_X)\in \mathcal{F}$ and $\varphi:Z\twoheadrightarrow X$ a parametrization of $X$. Then, for all $k\in\N$, $\dk{k}(\varphi^\ast F_X)=\dk{k}(F_X).$
\end{corollary}

\subsubsection{Common parametrizations of two spaces: tripods.}
Now, given $\bfX = (X,F_X)$ and $\bfY=(Y,F_Y)$ in $\mathcal{F}$, the main idea in comparing filtrations defined on different spaces is to consider parametrizations 
$\varphi_X:Z\twoheadrightarrow X$ and $\varphi_Y:Z\twoheadrightarrow Y$ of $X$ and $Y$ from a \emph{common} parameter space $Z$, i.e. \emph{tripods}:

$$
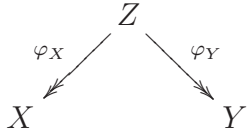
\xymatrix{ & Z \ar@{->>}[dl]_{{\varphi_X}} \ar@{->>}[dr]^{\varphi_Y} &  \\ X	& & Y}$$

and compare the pullback filtrations  $\varphi_X^*F_X$ and $\varphi_Y^*F_Y$ on $Z$. Formally, define

 \begin{multline}
 d_{\mathcal{F}}\big(\bfX,\bfY\big):=\\\inf\left\{\max_{\tau\in\mathrm{pow}(Z)}\big|\varphi^\ast_X F_X(\tau)-\varphi^\ast_Y F_Y(\tau)\big|;\,\varphi_X:Z\twoheadrightarrow X,\, \varphi_Y:Z\twoheadrightarrow Y\,\,\mbox{parametrizations}\right\}.
 \end{multline}

\begin{remark} \label{rem:dist-F-simple} Notice that in case $X=\{\ast\}$ and $F_{\{\ast\}}(\ast) = c\in\R$, then $d_{\mathcal{F}}(X,Y)=\max_{\sigma\subset Y}\big|F_Y(\sigma)-c\big|,$ for any filtered space $Y$. If $c=0$,  $Y=\{y_1,y_2\}$ with $F_Y(y_1)=F_Y(y_2)=0$ and $F_{Y}(\{y_1,y_2\})=1$. Then, $d_{\mathcal{F}}(X,Y)=1.$

However, still with $c=0$ and $Y=\{y_1,y_2\}$,  but $F_Y(y_1)=F_Y(y_2)=F_{Y}(\{y_1,y_2\})=0$, one has $d_\mathcal{F}(X,Y)=0$. This means that $d_{\mathcal{F}}$ is at best a pseudometric on filtered spaces.
\end{remark}

\begin{proposition}
$d_{\mathcal{F}}$ is a pseudometric on $\mathcal{F}$.
\end{proposition}
\begin{proof}
Symmetry and non-negativity are clear. We need to prove the triangle inequality. Let $\bfX = (X,F_X)$, $\bfY = (Y,F_Y)$, and $\bfW = (W,F_W)$ in $\mathcal{F}$ be non-empty and $\eta_1,\eta_2>0$ be s.t. 
$$d_{\mathcal{F}}(\bfX,\bfY)<\eta_1\,\,\mbox{and}\,\,d_{\mathcal{F}}(\bfY,\bfW)<\eta_2.$$ 
Choose, $\psi_X:Z_1\twoheadrightarrow X$, $\psi_Y:Z_1\twoheadrightarrow Y$, $\zeta_Y:Z_2\twoheadrightarrow Y$, and $\zeta_W:Z_2\twoheadrightarrow W$ surjective such that 
$$\|F_X\circ\psi_X-F_Y\circ\psi_Y\|_{\ell^\infty(\pow{Z_1})}<\eta_1$$
and
$$\|F_Y\circ\zeta_Y-F_W\circ\zeta_W\|_{\ell^\infty(\pow{Z_2})}<\eta_2.$$
Let $Z\subseteq Z_1\times Z_2$ be defined by $Z:=\{(z_1,z_2)\in Z_1\times Z_2|\psi_Y(z_1)=\zeta_Y(z_2)\}$ and consider the following (pullback) diagram:

$$
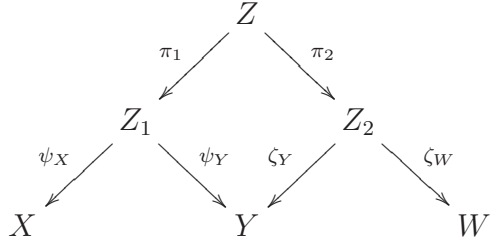
\xymatrix{& & Z \ar[dl]_{{\pi_1}} \ar[dr]^{\pi_2} & &  \\ & Z_1\ar[dl]_{\psi_X} \ar[dr]^{\psi_Y}& & Z_2\ar[dl]_{\zeta_Y}\ar[dr]^{\zeta_W} &\\ X & & Y & & W}.$$

Clearly, since $\psi_Y$ and $\zeta_Y$ are surjective, $Z$ is non-empty.
Now, consider the following three maps with domain $Z$: $\phi_X := \psi_X\circ\pi_1$, $\phi_Y := \psi_Y\circ\pi_1=\zeta_Y\circ\pi_2$, and  $\phi_W:=\zeta_W\circ\pi_2$. These three maps are surjective and therefore constitute parametrizations of $X$, $Y$, and $W$, respectively.  Then, since $\pi_i:Z\rightarrow Z_i$, $i=1,2$, are surjective and $\psi_Y\circ {\pi_1} = \zeta_Y\circ\pi_2$, we have
\begin{align*}
d_{\mathcal{F}}(\bfX,\bfW)&\leq \|F_X\circ\phi_X-F_W\circ\phi_W\|_{\ell^\infty(\pow{Z})}\\
&\leq \|F_X\circ\phi_X - F_Y\circ\phi_Y\|_{\ell^\infty(\pow{Z})}+\|F_Y\circ\phi_Y - F_W\circ\phi_W\|_{\ell^\infty(\pow{Z})}\\
&= \|F_X\circ\psi_X - F_Y\circ\psi_Y\|_{\ell^\infty(\pow{Z_1})}+\|F_Y\circ\zeta_Y - F_W\circ\zeta_W\|_{\ell^\infty(\pow{Z_2})}\\
&\leq \eta_1+\eta_2.
\end{align*}
The conclusion follows by letting $\eta_1\searrow d_{\mathcal{F}}(X,Y)$ and $\eta_2\searrow d_{\mathcal{F}}(Y,W)$.
\end{proof}

We now obtain a lifted version of Theorem \ref{theo:stab-vineyards}
 \begin{theorem} \label{theo:stab-pullback}
 For all finite filtered spaces $\bfX  =(X,F_X)$ and $\bfY = (Y,F_Y)$, and all $k\in\N$ one has:
 $$d_{\mathcal{D}}(\dk{k}(\bfX),\dk{k}(\bfY))\leq d_{\mathcal{F}}(\bfX,\bfY).$$
 \end{theorem}

 \begin{proof}[Proof of Theorem \ref{theo:stab-pullback}]
 Assume $\varepsilon>0$ is such that $d_{\mathcal{F}}(F_X,F_Y)<\varepsilon.$ Then, let $\varphi_X:Z\rightarrow X$ and $\varphi_Y:Z\rightarrow Y$ be surjective maps from the finite set $Z$ into $X$ and $Y$, respectively, such that $|\varphi_X^\ast F_X(\tau) - \varphi_Y^\ast F_Y(\tau)|<\varepsilon$ for all $\tau\in\mathrm{pow}(Z)$. Then, by Theorem \ref{theo:stab-vineyards}, 
 $$d_{\mathcal{D}}(\dk{k}(\varphi_X^\ast F_X),\dk{k}(\varphi_Y^\ast F_Y))<\varepsilon.$$
 for all $k\in\N.$ Now apply Corollary \ref{coro:same} and conclude by letting $\varepsilon$ approach $d_{\mathcal{F}}(\bfX,\bfY)$.
 \end{proof}

\begin{remark}\label{rem:non-tight}
Consider the case of $\mathbf{Y}$ being the one point filtered space $\{\ast\}$ such that $F_{\{\ast\}}(\{\ast\})=0$, and    $\mathbf{X}$ such that $X=\{x_2,x_2\}$, and $F_X(\{x_1\})=F_X(\{x_2\})=0$, $F_X(\{x_1,x_2\})=1$. In this case $d_{\mathcal{F}}(\mathbf{X},\mathbf{Y})=1$. However, notice that for $k=0$ $\mathrm{D}_0(\mathbf{X}) = \{[0,\infty),[0,1)\}$
 and $\mathrm{D}_0(\mathbf{Y}) = \{[0,\infty)\}$. Additionaly, for all $k\geq 1$ one has $\mathrm{D}_k(\mathbf{X}) = \mathrm{D}_k(\mathbf{Y}) = \emptyset$. This means that the lower bound provided by Theorem \ref{theo:stab-pullback} is equal to $\frac{1}{2}<1=d_{\mathcal{F}}(\mathbf{X},\mathbf{Y})$.
\end{remark}

\section{Filtrations arising from metric spaces: Rips and \v{C}ech}
Recall \cite{burago-book} that for two compact metric spaces $(X,d_X)$ and $(Y,d_Y)$, a correspondence between them is amy subset $R$ of $X\times Y$ such that the natural projections $\pi_X:X\times Y\rightarrow X$ and  $\pi_Y:X\times Y\rightarrow Y$ are such that $\pi_X(R)=X$ and $\pi_Y(R)=Y$. The distortion of any such correspondence is given by 
$$\mathrm{dis}(R):=\sup_{(x,y),(x',y')\in R}\big|d_X(x,x')-d_Y(y,y')\big|.$$
Then, Gromov-Hausdorff distance between $(X,d_X)$ and $(Y,d_Y)$ is defined as 
$$\dgro{X}{Y} := \frac{1}{2}\inf_{R}\mathrm{dis}(R),$$
where the infimum is taken over all correspondences $R$ between $X$ and $Y$.

\subsection{The Rips filtration}
Recall the definition of the \emph{Rips filtration} of a finite metric space $(X,d_X)$: for $\sigma\in \pow{X}$,
$$F^{\mathrm{R}}_X(\sigma)=\diamms{\sigma}{X}:=\max_{x,x'\in X}d_X(x,x').$$

The following theorem was first proved in \cite{dgw-topo-pers}. A different proof (also applicable to compact metric spaces) relying on the interleaving distance and multivalued maps was given in \cite{chazal-geom}. Yet another different proof avoiding multivalued maps is given in \cite{dowker-ph}.

\begin{theorem}\label{theo:stab-dD-R}
For all finite metric spaces $X$ and $Y$, and all $k\in\N$,
$$d_{\mathcal{D}}\big(\dk{k}(F^{\mathrm{R}}_X),\dk{k}(F^{\mathrm{R}}_Y)\big)\leq 2\,\dgro{X}{Y}.$$
\end{theorem}

A different proof of Theorem \ref{theo:stab-dD-R} can be obtained by combining Theorem \ref{theo:stab-pullback} and Proposition \ref{prop:stab-R} below.
\begin{proposition}\label{prop:stab-R}
For all finite metric spaces $X$ and $Y$, 
$$d_{\mathcal{F}}\big(F^{\mathrm{R}}_X,F^{\mathrm{R}}_Y\big)\leq 2\,\dgro{X}{Y}.$$
\end{proposition}
\begin{proof}[Proof of Proposition \ref{prop:stab-R}]
Let $X$ and $Y$ be s.t. $\dgro{X}{Y}<\eta$, and let $R\subset X\times Y$ be a surjective relation with $|d_X(x,x')-d_Y(y,y')|\leq 2\eta$ for all $(x,y),(x',y')\in R$. Consider the parametrization $Z=R$, and $\varphi_X=\pi_1:Z\rightarrow X$ and $\varphi_Y=\pi_2:Z\rightarrow Y$, then 
\beq{eq:param}
|d_X(\varphi_X(t),\varphi_X(t'))-d_Y(\varphi_Y(t),\varphi_Y(t'))|\leq 2\eta
\eeq
for all $t,t'\in Z$. Pick any $\tau\in Z$ and notice that

$$\varphi^*_XF_X^{\mathrm{R}}(\tau) = F_X^{\mathrm{R}}(\varphi_X(\tau)) = \max_{t,t'\in \tau}d_X(\varphi_X(t),\varphi_X(t')).$$

Now, similarly, write 
\begin{multline}\varphi^*_YF_Y^{\mathrm{R}}(\tau) = \max_{t,t'\in \tau}d_Y(\varphi_Y(t),\varphi_Y(t'))\leq \max_{t,t'\in\tau}d_X(\varphi_X(t),\varphi_X(t'))+2\eta = \varphi^*_XF_X^{\mathrm{R}}(\tau) + 2\eta,
\end{multline}

where the last inequality follows from \refeq{eq:param}. The proof follows by interchanging the roles of $X$ and $Y$.
\end{proof}

\subsection{The \v{C}ech filtration}
Another interesting and frequently used filtration is the \emph{\v{C}ech filtration}: for each $\sigma\in\pow{X}$,  
$$F^\mathrm{C}_X(\sigma) := \mathbf{rad}_X(\sigma)=\min_{p\in X}\max_{x\in \sigma}d_X(x,p).$$ That is, the filtration value of each simplex corresponds to its \emph{circumradius}.
\begin{proposition}\label{prop:stab-C}
For all finite metric spaces $X$ and $Y$, 
$$d_{\mathcal{F}}\big(F^{W}_X,F^{W}_Y\big)\leq 2\,\dgro{X}{Y}.$$
\end{proposition}

Again, as a corollary of Theorem \ref{theo:stab-pullback} and Proposition \ref{prop:stab-C} we have the following 
\begin{theorem}\label{theo:stab-dD-C}
For all finite metric spaces $X$ and $Y$, and all $k\in\N$,
$$d_{\mathcal{D}}\big(\dk{k}(F^{\mathrm{C}}_X),\dk{k}(F^{\mathrm{C}}_Y)\big)\leq 2\,\dgro{X}{Y}.$$
\end{theorem}
A proof of this theorem via the interleaving distance and multi-valued maps has appeared in \cite{chazal-geom}.\footnote{The version in \cite{chazal-geom} applies to compact metric spaces.} Another proof avoiding multivalued maps is given in \cite{dowker-ph}.

\begin{proof}[Proof of Proposition \ref{prop:stab-C}]

The proof is similar to that of Proposition \ref{prop:stab-R}. Pick any $\tau\in Z$, then,

$$\varphi^*_XF_X^W(\tau) = F_X^W(\varphi_X(\tau)) = \min_{p\in X}\max_{t\in \tau}d_X(p,\varphi_X(t))=\max_{t\in \tau}d_X(p_\tau,\varphi_X(t))$$
for some $p_\tau\in X$. Let $t_\tau\in Z$ be s.t. $\varphi_X(t_\tau)=p_\tau$, and from the above obtain
$$\varphi^*_XF_X^W(\tau) = \max_{t\in \tau}d_X(\varphi_X(t_\tau),\varphi_X(t)).$$

Now, similarly, write 
\begin{multline}\varphi^*_YF_Y^W(\tau) = \min_{q\in Y}\max_{t\in \tau}d_Y(q,\varphi_Y(t))\leq \max_{t\in\tau}d_Y(\varphi_Y(t_\tau),\varphi_Y(t))\leq \max_{t\in\tau}d_X(\varphi_X(t_\tau),\varphi_X(t))+2\eta\\ = \varphi^*_XF_X^W(\tau) + 2\eta,
\end{multline}

where the last inequality follows from \refeq{eq:param}. The proof follows by interchanging the roles of $X$ and $Y$.
\end{proof}

\section{$d_{\mathcal{F}}$ is geodesic}\label{sec:geodesic}
In this section we construct geodesics between any pair $\mathbf{X}$ and $\mathbf{Y}$ of filtered spaces and obtain a strengthening of Theorem \ref{theo:stab-pullback}.

\subsection{Geodesics}
Given $\mathbf{X}$ and $\mathbf{Y}$ in $\mathcal{F}$ consider $\mathcal{T}^\mathrm{opt}(\mathbf{X}.\mathbf{Y})$ the set of all minimizing tripods: That is, for any $(Z,\varphi_X,\varphi_Y)\in\mathcal{T}(\mathbf{X},\mathbf{Y})$ we have $\|\varphi_X^\ast F_X-\varphi_Y^*F_Y\|_{\ell^\infty(\pow{Z}} = d_{\mathcal{F}}(\mathbf{X},\mathbf{Y}).$

For each minimizing tripod $T=(Z,\varphi_X,\varphi_Y)\in\mathcal{T}^\mathrm{opt}(\mathbf{X},\mathbf{Y})$ consider the curve $$\mbox{$\gamma_T:[0,1]\rightarrow \mathcal{F}$ defined by $t\mapsto \mathbf{Z_t}:=(Z,F_t)$}$$ where 
 $$F_t:=(1-t)\cdot \varphi_X^*F_X+t\cdot \varphi_Y^* F_Y.$$

\begin{theorem}
For each $T\in\mathcal{T}^\mathrm{opt}(\mathbf{X},\mathbf{Y})$ the curve $\gamma_T$ is a geodesic between $\mathbf{X}$ and $\mathbf{Y}$. Namely, for all $s,t\in[0,1]$ one has:
$$d_{\mathcal{F}}(\gamma_T(s),\gamma_T(t))=|s-t|\cdot d_{\mathcal{F}}(\mathbf{X},\mathbf{Y}).$$
\end{theorem}
\begin{proof}
Let $\eta = d_{\mathcal{F}}(\mathbf{X},\mathbf{Y})$. We check that $$(\ast )\,\,\,\,d_{\mathcal{F}}(\gamma_T(s),\gamma_T(t))\leq |s-t|\cdot \eta$$ and notice that this is enough. Otherwise, let $s<t$ in $[0,1]$ be such that $d_{\mathcal{F}}(\gamma_T(s),\gamma_T(t))\leq (t-s)\cdot \eta$. Then,   by the triangle inequality for $d_{\mathcal{F}}$ we would have
$$\eta \leq d_{\mathcal{F}}(\gamma_T(0),\gamma_T(s))+ d_{\mathcal{F}}(\gamma_T(s),\gamma_T(t))+ d_{\mathcal{F}}(\gamma_T(t),\gamma_T(1))$$
which by $(*)$ and the definition of $s$ and $t$ would be \emph{strictly} smaller than $\eta\cdot s+(t-s)\cdot \eta + (1-t)\cdot \eta = \eta$, a contradiction. 

Now, in order to verify $(\ast)$, we need to construct a tripod between $\gamma_T(s) = (Z,F_s)$ and $\gamma_T(t) = (Z,F_t)$. We consider the tripod $(Z,\mathrm{id}_Z,\mathrm{id}_Z)$ and notice that this tripod gives
\begin{align}
d_{\mathcal{F}}(\gamma_T(s),\gamma_T(t))&\leq \max_{\tau\in\pow{Z}}\big|F_s(\tau)-F_t(\tau)\big|\\
&= \max_{\tau\in\pow{Z}}\bigg|\big((1-s)-(1-t)\big)\cdot F_X(\varphi_X(\tau)) - (t-s)\cdot F_Y(\varphi_Y(\tau))\bigg|\\
&=|t-s|\cdot \max_{\tau\in\pow{Z}}| F_X(\varphi_X(\tau)) - F_Y(\varphi_Y(\tau))|\\
&=|t-s|\cdot \eta.
\end{align}
\end{proof}

\subsection{A strengthening of Theorem \ref{theo:stab-pullback}}
Recall the definition of \emph{length} in a metric space $(M,d_M)$. For a curve $\alpha:[0,1]\rightarrow M$, its length is
$$\mathrm{L}_M(\alpha):= \sup\left\{\sum_{i=0}^{n-1}d_M(\alpha(t_i),\alpha(t_{i+1}))|\,0=t_0<t_1<\cdots t_n=1\right\}.$$

We now use the construction of geodesics above to strengthen Theorem \ref{theo:stab-pullback}.

\begin{theorem}\label{theo:strengthening}
For all $\mathbf{X},\mathbf{Y}\in\mathcal{F}$ and $k\in\mathbb{N}$ one has   
$$\sup_{T\in\mathcal{T}^\mathrm{opt}} \mathrm{L}_\mathcal{D}(\mathrm{D}_k(\gamma_T)) \leq d_\mathcal{F}(\mathbf{X},\mathbf{Y}).$$
\end{theorem}

\begin{remark}[Strengthening] That this theorem is a strengthening can be argued as follows. Firstly, by definition of length  $\mathrm{L}_\mathcal{D}(\mathrm{D}_k(\gamma_T))\geq d_\mathcal{D}(\mathrm{D}_k(\gamma_T(0)),\mathrm{D}_k(\gamma_T(1))) = d_\mathcal{D}(\mathrm{D}_k(\mathbf{X}),\mathrm{D}_k(\mathbf{Y})).$ Therefore the lower bound provided by Theorem \ref{theo:strengthening} cannot be smaller than the one provided by Theorem \ref{theo:stab-pullback}.

Now, to argue about the improvement offered by the new lower bound consider the two spaces from Remark \ref{rem:non-tight}. For those spaces, a minimizing tripod is $T=(X,\mathrm{id}_X,\varphi)$ where $\varphi:X\rightarrow \{\ast\}$ is the unique map. In that case, one sees that $F_t = (1-t)\cdot F_X$ and therefore $\alpha_T(t):=\mathrm{D}_0(\gamma_T(t)) = \{[0,\infty),[0,(1-t))\}.$ Notice that for ant two $t,s\in[0,1]$ such that $|s-t|$ is sufficiently small, the bottleneck distance $d_{\mathcal{D}}(\alpha_T(t),\alpha_T(s))$ equals $|t-s|.$ Then, $\mathrm{L}_{\mathcal{D}}(\alpha) = 1=d_{\mathcal{F}}(\gamma_T(0),\gamma_T(1)) >\frac{1}{2}=d_{\mathcal{D}}(\alpha_T(0),\alpha_T(1)),$ which is the bound provided by Theorem \ref{theo:stab-pullback}.

 \end{remark}

\begin{proof}[Proof of Theorem \ref{theo:strengthening}]
Let $T\in \mathcal{T}^\mathrm{opt}(\mathbf{X},\mathbf{Y})$ and let $0=t_0<\cdots <t_n=1$ be any partition of $[0,1].$ Then, write
\begin{align*}
\sum_{i=0}^{n-1} d_\mathcal{D}\big(\mathrm{D}_k(\gamma_T(t_i)), \mathrm{D}_k(\gamma_T(t_{i+1}))\big) &\leq \sum_{i=1}^{n-1} d_\mathcal{F}\big(\gamma_T(t_i), \gamma_T(t_{i+1})\big)\\
&=\sum_{i=0}^{n-1} (t_{i+1}-t_i)\cdot d_{\mathcal{F}}(\gamma_{T}(0),\gamma_T(1))\\
&=d_{\mathcal{F}}(\gamma_{T}(0),\gamma_T(1))\\
&= d_{\mathcal{F}}(\mathbf{X},\mathbf{Y}),
\end{align*}
where the first inequality follows from Theorem \ref{theo:stab-pullback}, and the equality immediately after it follows from the fact that $\gamma_T$ is a geodesic.

Thus, we have for any partition $0=t_0<\cdot <t_1=1$ that 
$$\sum_{i=0}^{n-1} d_\mathcal{D}\big(\mathrm{D}_k(\gamma_T(t_i)), \mathrm{D}_k(\gamma_T(t_{i+1}))\big)\leq d_{\mathcal{F}}(\mathbf{X},\mathbf{Y}).$$
Taking supremum over all possible partitions yields the claim.
\end{proof}

\begin{remark}
The techniques of the above theorem and the results in \cite{chowdhury2016constructing} imply similar strengthenings of Propositions \ref{prop:stab-R} and \ref{prop:stab-C}.
\end{remark}

\section{Discussion}
It seems possible to extend some of these ideas to the case of non necessarily finite filtered spaces.

\newcommand{\etalchar}[1]{$^{#1}$}

\end{document}